\documentclass{article}
\usepackage{amsmath}
\usepackage{amssymb}
\usepackage{amsfonts}

\newtheorem{theorem}{Theorem}

\newtheorem{lemma}[theorem]{Lemma}

\newenvironment{proof}[1][Proof]{\noindent \textbf{#1.} }{\  \rule{0.5em}{0.5em}}
\input{tcilatex}
\begin{document}

\author{Zafer \c{S}iar$^{a}$ and Refik Keskin$^{b}$ \and $\ ^{a}$Bing\"{o}l
University, Department of Mathematics, Bing\"{o}l/TURKEY \and $^{b}$Sakarya
University, Department of Mathematics, Sakarya/TURKEY \and $^{a}$%
zsiar@bingol.edu.tr, $^{b}$rkeskin@sakarya.edu.tr}
\title{On Perfect Powers in $k$-Generalized Pell-Lucas Sequence}
\maketitle

\begin{abstract}
Let $k\geq 2$ and let $(Q_{n}^{(k)})_{n\geq 2-k}$ be the $k$-generalized
Pell sequence defined by 
\begin{equation*}
Q_{n}^{(k)}=2Q_{n-1}^{(k)}+Q_{n-2}^{(k)}+\cdots +Q_{n-k}^{(k)}
\end{equation*}%
for $n\geq 2$ with initial conditions 
\begin{equation*}
Q_{-(k-2)}^{(k)}=Q_{-(k-3)}^{(k)}=\cdots
=Q_{-1}^{(k)}=0,~Q_{0}^{(k)}=2,Q_{1}^{(k)}=2.
\end{equation*}

In this paper, we solve the Diophantine equation 
\begin{equation*}
Q_{n}^{(k)}=y^{m}
\end{equation*}%
in positive integers $n,m,y,k$ with $m,y,k\geq 2$. We show that all
solutions $(n,m,y)$ of this equation in positive integers $n,m,y,k$ such
that $2\leq y\leq 100$ are given by $(n,m,y)=(3,2,4),(3,4,2)$ for $k\geq 3$.
Namely, $Q_{3}^{(k)}=16=2^{4}=4^{2}$ for $k\geq 3.$
\end{abstract}

\bigskip Keywords: Fibonacci and Lucas numbers, Exponential Diophantine
equations, Linear forms in logarithms; Baker's method

AMS Subject Classification(2010): 11B39, 11D61, 11J86,

\section{\protect\bigskip Introduction}

Let $k,r$ be an integer with $k\geq 2$ and $r\neq 0$. Let the linear
recurrence sequence $\left( G_{n}^{(k)}\right) _{n\geq 2-k}$ of order $k$ be
defined by 
\begin{equation}
G_{n}^{(k)}=rG_{n-1}^{(k)}+G_{n-2}^{(k)}+\ldots +G_{n-k}^{(k)}\text{ }
\label{0}
\end{equation}%
for $n\geq 2$ with the initial conditions $%
G_{-(k-2)}^{(k)}=G_{-(k-3)}^{(k)}=\cdots =G_{-1}^{(k)}=0,$ $G_{0}^{(k)}=a,$
and $G_{1}^{(k)}=b.$ For $(a,b,r)=(0,1,1),~$the sequence $\left(
G_{n}^{(k)}\right) _{n\geq 2-k}$ is called $k$-generalized Fibonacci
sequence $\left( F_{n}^{(k)}\right) _{n\geq 2-k}$ (see \cite{luca2}). For $%
(a,b,r)=(0,1,2)$ and $(a,b,r)=(2,2,2),~$the sequence $\left(
G_{n}^{(k)}\right) _{n\geq 2-k}$ is called $k$-generalized Pell sequence $%
\left( P_{n}^{(k)}\right) _{n\geq 2-k}$ and $k$-generalized Pell-Lucas
sequence $\left( Q_{n}^{(k)}\right) _{n\geq 2-k},$ respectively (see \cite%
{klc1}). The terms of these sequences are called $k$-generalized Fibonacci
numbers, $k$-generalized Pell numbers and $k$-generalized Pell-Lucas
numbers, respectively. When $k=2,$ we have Fibonacci, Pell and Pell-Lucas
sequences, $\left( F_{n}\right) _{n\geq 0},$ $\left( P_{n}\right) _{n\geq
0}, $ and $\left( Q_{n}\right) _{n\geq 0}$, respectively.

There has been much interest in when the terms of linear recurrence
sequences are perfect powers. For instance, in \cite{ljung}, Ljunggren
showed that for $n\geq 2$, $P_{n}$ is a perfect square precisely for $%
P_{7}=13^{2}$ and $P_{n}=2x^{2}$ precisely for $P_{2}=2$. In \cite{cohn},
Cohn solved the same equations for Fibonacci numbers. Later, these problems
are extended by Peth\`{o} for Pell numbers and by Bugeaud, Mignotte and
Siksek for Fibonacci numbers. Peth\`{o} \cite{petho} and Cohn \cite{cohn1}
independenty found all perfect powers in the Pell sequence. They proved that
the only positive integer solution $(n,y,m)$ with $m\geq 2$ and $y\geq 2$ of
the Diophantine equation $P_{n}=y^{m}$ is given by $(n,y,m)=(7,13,2)$.
Bugeaud, Mignotte and Siksek \cite{Bgud} solved the Diophantine equation $%
F_{n}=y^{p}$ for $p\geq 2$ using modular approach and classical linear forms
in logarithms. Bravo and Luca showed in \cite{luca2} that the Diophantine
equation $F_{n}^{(k)}=2^{m}$ in positive integers $n,k,m$ with $k\geq 2$ has
the solutions $(n,k,m)=(6,2,3),(1,k,0)$ and $(n,k,m)=(t,k,t-2)$ for all $%
2\leq t\leq k+1.$ Except these, recently, for the studies related to $k$%
-generalized Fibonacci sequence, one can consult \cite{luca4,luca5,togbe1}.

In \cite{bravo2}, the authors found all perfect powers in the Pell-Lucas
sequence and proved the following result:

\begin{theorem}
\label{ty}Let $n,y,m\in 
\mathbb{N}
$ with $m\geq 2$. Then the equation $Q_{n}=y^{m}$ has no integer solutions
and the equation $Q_{n}=2y^{m}$ has only the solution $Q_{1}=2\cdot 1^{m}$.
\end{theorem}

In this paper, we will handle the Diophantine equation 
\begin{equation}
Q_{n}^{(k)}=y^{m},~k,n,m,y\in 
\mathbb{Z}
^{+}\text{ with }k,y,m\geq 2.  \label{1.1}
\end{equation}%
We will show that the solutions $(n,m,y)$ of Diophantine equation (\ref{1.1}%
) with $2\leq y\leq 100$ are given by $(n,m,y)=(3,2,4),(3,4,2)$ for $k\geq 3$%
. Namely, $Q_{3}^{(k)}=16=2^{4}=4^{2}$ for $k\geq 3.$

\section{Preliminaries}

In this section, we will mention some facts and properties of $k-$%
generalized Pell and Pell-Lucas sequences. It can be seen that the
characteristic polynomial of these sequences is%
\begin{equation}
\Psi _{k}(x)=x^{k}-2x^{k-1}-\cdots -x-1.  \label{0.2}
\end{equation}%
Let us denote the roots of the polynomial in (\ref{0.2}) by $\alpha _{j}$
for $j=1,2,\ldots ,k.$\ We know from Lemma 1 given in \cite{zang} that the
polynomial $\Psi _{k}(x)$ has exactly one positive real root located between 
$2$ and $3.$ Particuarly, let $\alpha =\alpha (k)=\alpha _{1},$ be the
positive real root of the polynomial $\Psi _{k}(x)$. So, 
\begin{equation}
2<\alpha <3.  \label{0.7}
\end{equation}%
The other roots are strictly inside the unit circle.

In \cite{luca3}, the Binet-like formula for the $k$-generalized Pell number
is given by 
\begin{equation*}
P_{n}^{(k)}=\dsum\limits_{j=1}^{k}\frac{(\alpha _{j}-1)}{\alpha
_{j}^{2}-1+k(\alpha _{j}^{2}-3\alpha _{j}+1)}\alpha _{j}^{n}.
\end{equation*}%
If we follow the method given in \cite{luca3} to obtain the Binet-like
formula for $k$-generalized Pell-Lucas numbers, then we get 
\begin{equation}
Q_{n}^{(k)}=\dsum\limits_{j=1}^{k}\frac{2(\alpha _{j}-1)^{2}}{\alpha
_{j}^{2}-1+k(\alpha _{j}^{2}-3\alpha _{j}+1)}\alpha _{j}^{n}.  \label{0.8}
\end{equation}%
From \cite{bravo1} \emph{,} we can give the following lemma, which will be
used in the proof of Theorem \ref{T4}.

\begin{lemma}
\label{L6}Let $\alpha _{j}$'s for $j=1,2,\ldots ,k,$ be the roots of $\Psi
_{k}(x),$ and let $\alpha =\alpha _{1}$ be dominant root of $\Psi _{k}(x).$
Then for $j\geq 1$ and $k\geq 2,$ the inequality 
\begin{equation}
\left \vert \frac{(\alpha _{j}-1)}{\alpha _{j}^{2}-1+k(\alpha
_{j}^{2}-3\alpha _{j}+1)}\right \vert <1  \label{3.4}
\end{equation}%
holds.
\end{lemma}

Throughout this paper, $\alpha $ denotes the positive root of the polynomial 
$\Psi _{k}(x)$. The relation between $\alpha $ and $P_{n}^{(k)}$ is given by 
\begin{equation}
\alpha ^{n-2}\leq P_{n}^{(k)}\leq \alpha ^{n-1}  \label{1.3}
\end{equation}%
for all $n\geq 1$. For a proof of (\ref{1.3}), see \cite{luca3}. Also, K\i
l\i \c{c} \cite{klc1} proved that 
\begin{equation}
P_{n}^{(k)}=F_{2n-1}  \label{1.2}
\end{equation}%
for all $1\leq n\leq k+1.$

Now, we prove a result concerning the relation between $P_{n}^{(k)}$ and $%
Q_{n}^{(k)}.$

\begin{lemma}
\label{L5}Let $k\geq 2$ be an integer. Then the relation 
\begin{equation}
Q_{n}^{(k)}=2\left( P_{n+1}^{(k)}-P_{n}^{(k)}\right)  \label{w}
\end{equation}%
always holds.
\end{lemma}

\begin{proof}
Let 
\begin{equation*}
G(x)=\dsum \limits_{i=0}^{\infty }G_{i}^{(k)}x^{i}
\end{equation*}%
be the generating function for $G_{n}^{(k)}$ given by the relation (\ref{0}%
). Using the relation (\ref{0}) with initial conditons, it can be seen that 
\begin{equation*}
G(x)=\frac{a+(b-2a)x}{1-2x-x^{2}-\cdots -x^{k}}
\end{equation*}%
for $r=2.$ So, the generating function $P(x)$ for $P_{n}^{(k)}$ is 
\begin{equation}
P(x)=\frac{x}{1-2x-x^{2}-\cdots -x^{k}}  \label{0.9}
\end{equation}%
and the generating function $Q(x)$ for $Q_{n}^{(k)}$ is 
\begin{equation}
Q(x)=\frac{2-2x}{1-2x-x^{2}-\cdots -x^{k}}.  \label{0.10}
\end{equation}%
From (\ref{0.9}) and (\ref{0.10}), the proof follows.
\end{proof}

Now, we will give some theorems and lemmas from \cite{luca3}, which will be
useful in the next section.

\begin{theorem}
\label{T1}\emph{(\cite{luca3}, Theorem 3.1)} Let $k\geq 2$ be an integer.
Then, for all $n\geq 2-k,$ we have 
\begin{equation*}
P_{n}^{(k)}=\dsum \limits_{i=1}^{k}g_{k}(\alpha _{i})\alpha _{i}^{n}\text{ }
\end{equation*}%
and%
\begin{equation*}
\left \vert P_{n}^{(k)}-g_{k}(\alpha )\alpha ^{n}\right \vert <\frac{1}{2},
\end{equation*}%
where $\alpha =\alpha _{1},\alpha _{2,}...,\alpha _{k}$ are the roots of the
characteristic equation $\Psi _{k}(x)=0$ and 
\begin{equation}
g_{k}(z)=\frac{z-1}{(k+1)z^{2}-3kz+k-1}.  \label{0.5}
\end{equation}
\end{theorem}

Thus, from (\ref{0.8}), we can write%
\begin{equation}
Q_{n}^{(k)}=\dsum \limits_{i=1}^{k}(2\alpha _{i}-2)g_{k}(\alpha _{i})\alpha
_{i}^{n}.  \label{0.11}
\end{equation}

\begin{lemma}
\label{L1}\emph{(\cite{luca3}, Lemma 3.2) }Let $k,l\geq 2$ be integers. Then

\emph{(a)} If $k>l,$ then $\alpha (k)>\alpha (l),$ where $\alpha (k)$ and $%
\alpha (l)$ are the values of $\alpha $ relative to $k$ and $l,$
respectively.

\emph{(b)} $\varphi ^{2}(1-\varphi ^{-k})<\alpha <\varphi ^{2},$ where $%
\varphi =\frac{1+\sqrt{5}}{2}$ is the golden section.

\emph{(c)}$~g_{k}(\varphi ^{2})=\frac{1}{\varphi +2}.$

\emph{(d)} $0.276<g_{k}(\alpha )<0.5.$

\emph{(e)} If $k\geq 6,$ then 
\begin{equation*}
c_{k}<\alpha <\varphi ^{2},
\end{equation*}%
where 
\begin{equation*}
c_{k}=\frac{3k+\sqrt{5k^{2}+4}}{2k+2}.
\end{equation*}
\end{lemma}

It is easy to observe that the inequality in Lemma \ref{L1} (e) holds for $%
k\geq 2.$ If we consider the function $g_{k}(x)$ defined in (\ref{0.5}) as a
function of a real variable, then it can be easily seen that the function $%
g_{k}(x)$ is decreasing and continuous in the interval $(c_{k},\infty )$
(also see Lemma 3.1 in \cite{luca3}). Therefore, by Lemma \ref{L1}, we have 
\begin{equation}
g_{k}(\varphi ^{2})<g_{k}(t)<g_{k}(\alpha )<\frac{1}{2}  \label{0.6}
\end{equation}%
for every $t\in (\alpha ,\varphi ^{2}).$ The inequality (\ref{0.6}) implies
that 
\begin{equation}
(k+1)t^{2}-3kt+k-1>2  \label{0.12}
\end{equation}%
for $t\in (\alpha ,\varphi ^{2}).$

For solving the equation (\ref{1.1}), we use linear forms in logarithms and
Baker's theory. For this, we will give some notions, theorem, and lemmas
related to linear forms in logarithms and Baker's Theory.

Let $\eta $ be an algebraic number of degree $d$ with minimal polynomial 
\begin{equation*}
a_{0}x^{d}+a_{1}x^{d-1}+\cdots +a_{d}=a_{0}\dprod\limits_{i=1}^{d}\left(
x-\eta ^{(i)}\right) \in \mathbb{Z}[x],
\end{equation*}%
where the $a_{i}$'s are integers with $\gcd (a_{0},\ldots ,a_{n})=1$ and $%
a_{0}>0$ and the $\eta ^{(i)}$'s are conjugates of $\eta .$ Then 
\begin{equation}
h(\eta )=\frac{1}{d}\left( \log a_{0}+\dsum\limits_{i=1}^{d}\log \left( \max
\left\{ |\eta ^{(i)}|,1\right\} \right) \right)  \label{2.1}
\end{equation}%
is called the logarithmic height of $\eta .$ In particular, if $\eta =a/b$
is a rational number with $\gcd (a,b)=1$ and $b\geq 1,$ then $h(\eta )=\log
\left( \max \left\{ |a|,b\right\} \right) .$

We give some properties of the logarithmic height whose proofs can be found
in \cite{yann}:

\begin{equation}
h(\eta \pm \gamma )\leq h(\eta )+h(\gamma )+\log 2,  \label{2.2}
\end{equation}%
\begin{equation}
h(\eta \gamma ^{\pm 1})\leq h(\eta )+h(\gamma ),  \label{2.3}
\end{equation}%
\begin{equation}
h(\eta ^{m})=|m|h(\eta ).  \label{2.4}
\end{equation}%
Now, we can deduce the following estimation for $h(g_{k}(\alpha ))$ from
Lemma 6\ given in \cite{bravo3}.

\begin{lemma}
\label{L7}Let $k\geq 2.$ Then $h(g_{k}(\alpha ))<5\log k.$
\end{lemma}

The following theorem is deduced from Corollary 2.3 of Matveev \cite{Mtv}
and provides a large upper bound for the subscript $n$ in the equation (\ref%
{1.1}) (also see Theorem 9.4 in \cite{Bgud}).

\begin{theorem}
\label{T2} Assume that $\gamma _{1},\gamma _{2},\ldots ,\gamma _{t}$ are
positive real algebraic numbers in a real algebraic number field $\mathbb{K}$
of degree $D$, $b_{1},b_{2},\ldots ,b_{t}$ are rational integers, and 
\begin{equation*}
\Lambda :=\gamma _{1}^{b_{1}}\cdots \gamma _{t}^{b_{t}}-1
\end{equation*}%
is not zero. Then 
\begin{equation*}
|\Lambda |>\exp \left( -1.4\cdot 30^{t+3}\cdot t^{4.5}\cdot D^{2}(1+\log
D)(1+\log B)A_{1}A_{2}\cdots A_{t}\right) ,
\end{equation*}%
where 
\begin{equation*}
B\geq \max \left \{ |b_{1}|,\ldots ,|b_{t}|\right \} ,
\end{equation*}%
and $A_{i}\geq \max \left \{ Dh(\gamma _{i}),|\log \gamma
_{i}|,0.16\right
\} $ for all $i=1,\ldots ,t.$
\end{theorem}

Now we give a lemma which was proved in \cite{Bravo}. It is a version of the
lemma given by Dujella and Peth\H{o} \cite{duj}. The lemma given in \cite%
{duj} is a variation of a result of Baker and Davenport \cite{Baker}. This
lemma will be used to reduce the upper bound for the subscript $n$ in the
equation (\ref{1.1}). For any real number $x,$ we let $||x||=\min \left \{
|x-n|:n\in 
\mathbb{Z}
\right \} $ be the distance from $x$ to the nearest integer.

\begin{lemma}
\label{L2}Let $M$ be a positive integer, let $p/q$ be a convergent of the
continued fraction of the irrational number $\gamma $ such that $q>6M,$ and
let $A,B,\mu $ be some real numbers with $A>0$ and $B>1.$ Let $\epsilon
:=||\mu q||-M||\gamma q||.$ If $\epsilon >0,$ then there exists no solution
to the inequality 
\begin{equation*}
0<|u\gamma -v+\mu |<AB^{-w},
\end{equation*}%
in positive integers $u,v,$ and $w$ with 
\begin{equation*}
u\leq M\text{ and }w\geq \frac{\log (Aq/\epsilon )}{\log B}.
\end{equation*}
\end{lemma}

The following lemma can be found in \cite{weger}.

\begin{lemma}
\label{L3} Let $a,x\in 
\mathbb{R}
.$ If $0<a<1$ and $\left \vert x\right \vert <a,$ then 
\begin{equation*}
\left \vert \log (1+x)\right \vert <\frac{-\log (1-a)}{a}\cdot \left \vert
x\right \vert
\end{equation*}%
and 
\begin{equation*}
\left \vert x\right \vert <\frac{a}{1-e^{-a}}\cdot \left \vert e^{x}-1\right
\vert .
\end{equation*}
\end{lemma}

\section{Main Theorem}

Now, we prove a lemma, which will be used in the next theorem.

\begin{lemma}
\label{L4}Let $k\geq 2$ be an integer. Then

\emph{(a)} $\alpha ^{n-1}<Q_{n}^{(k)}<2\alpha ^{n}$ for all $n\geq 1.$

\emph{(b)} $\left \vert Q_{n}^{(k)}-(2\alpha -2)g_{k}(\alpha )\alpha
^{n}\right \vert <2$ for all $n\geq 2-k.$

\emph{(c)} $Q_{n}^{(k)}=2F_{2n}$ for $1\leq n\leq k.$
\end{lemma}

\begin{proof}
We have the relation $Q_{n}^{(k)}=2\left( P_{n+1}^{(k)}-P_{n}^{(k)}\right) $
by Lemma \ref{L5}.

(a) Using (\ref{1.3}), we get 
\begin{equation*}
Q_{n}^{(k)}=2\left( P_{n+1}^{(k)}-P_{n}^{(k)}\right) \leq 2\left( \alpha
^{n}-\alpha ^{n-2}\right) \leq 2\alpha ^{n-2}\left( \alpha ^{2}-1\right)
<2\alpha ^{n},
\end{equation*}%
and also 
\begin{equation*}
Q_{n}^{(k)}=2\left( P_{n+1}^{(k)}-P_{n}^{(k)}\right) \geq
P_{n+1}^{(k)}+P_{n-1}^{(k)}+P_{n-2}^{(k)}+\cdots +P_{n-k+1}^{(k)}>\alpha
^{n-1}.
\end{equation*}

(b) By using Theorem \ref{T1}, we obtain 
\begin{eqnarray*}
\left \vert Q_{n}^{(k)}-(2\alpha -2)g_{k}(\alpha )\alpha ^{n}\right \vert
&=&\left \vert 2P_{n+1}^{(k)}-2P_{n}^{(k)}-2g_{k}(\alpha )\alpha
^{n+1}+2g_{k}(\alpha )\alpha ^{n}\right \vert \\
&\leq &2\left \vert P_{n+1}^{(k)}-g_{k}(\alpha )\alpha ^{n+1}\right \vert
+2\left \vert P_{n}^{(k)}-g_{k}(\alpha )\alpha ^{n}\right \vert \\
&<&2.
\end{eqnarray*}

(c) From the equalities\ (\ref{1.2}) and (\ref{w}), the proof follows.
\end{proof}

\begin{theorem}
\label{T4}All solutions of Diophantine equation \emph{(\ref{1.1}) }satisfies
the inequality \textit{\ }%
\begin{equation}
n<1.64\cdot 10^{13}\cdot k^{4}\cdot (\log k)^{2}\cdot \log y\cdot \log n.
\label{3.5}
\end{equation}
\end{theorem}

\begin{proof}
Assume that the Diophantine equation (\ref{1.1}) holds. If $1\leq n\leq k,$
then we have $Q_{n}^{(k)}=2F_{2n}=y^{m}$ by (\ref{1.2}) and (\ref{w}). From
here, by Theorem 2 given in \cite{patel}, $F_{2n}=2^{m-1}(y/2)^{m}$ implies
that $n\leq 6$. Thus, the ldentity (\ref{3.5}) is satisfied. Then we suppose
that $n\geq k+1.$ In this case, $n\geq 3.$ Let $\alpha $ be positive real
root of $\Psi _{k}(x)$ given in (\ref{0.2}). Then $2<\alpha <\varphi ^{2}<3$
by (\ref{0.7}) and Lemma \ref{L1}. Using Lemma \ref{L4} (a), we get 
\begin{equation*}
\alpha ^{n-1}<y^{m}<2\alpha ^{n}.
\end{equation*}%
Making necessary calculations, we obtain 
\begin{equation}
m<\frac{\log 2}{\log y}+n\frac{\log \varphi ^{2}}{\log y}<1+n\frac{\log
\varphi ^{2}}{\log 2}<1.73n  \label{3.1}
\end{equation}%
for $n\geq 3$. Now, let us rearrange the equation (\ref{1.1}) by using Lemma %
\ref{L4} (b). Thus, we have 
\begin{equation}
\left\vert y^{m}-(2\alpha -2)g_{k}(\alpha )\alpha ^{n}\right\vert <2.
\label{3.2}
\end{equation}%
If we divide both sides of the inequality (\ref{3.2}) by $(2\alpha
-2)g_{k}(\alpha )\alpha ^{n},$ we get 
\begin{eqnarray}
\left\vert y^{m}\alpha ^{-n}((2\alpha -2)g_{k}(\alpha
))_{{}}^{-1}-1\right\vert &<&\frac{2}{(2\alpha -2)g_{k}(\alpha )\alpha ^{n}}%
\leq \frac{\alpha ^{-n}}{2\cdot 0.276}  \label{3.3} \\
&<&\frac{1.82}{\alpha ^{n}}  \notag
\end{eqnarray}%
by Lemma \ref{L1} (d). In order to use the result of Theorem \ref{T2}, we
take 
\begin{equation*}
\left( \gamma _{1},b_{1}\right) :=\left( y,m\right) ,~\left( \gamma
_{2},b_{2}\right) :=\left( \alpha ,-n\right) ,~\left( \gamma
_{3},b_{3}\right) :=\left( (2\alpha -2)g_{k}(\alpha ),-1\right) .
\end{equation*}%
The number field containing $\gamma _{1}$,$\gamma _{2}$, and $\gamma _{3}$
are $\mathbb{K}=%
\mathbb{Q}
(\alpha ),$ which has degree $D=k.$ We show that the number 
\begin{equation*}
\Lambda _{1}:=y^{m}\alpha ^{-n}((2\alpha -2)g_{k}(\alpha ))^{-1}-1
\end{equation*}%
is nonzero. Contrast to this, assume that $\Lambda _{1}=0.$ Then 
\begin{equation*}
y^{m}=(2\alpha -2)g_{k}(\alpha )\alpha ^{n}=\frac{(2\alpha -2)\left( \alpha
-1\right) }{(k+1)\alpha ^{2}-3k\alpha +k-1}\alpha ^{n}.
\end{equation*}%
Conjugating the above equality by some automorphisim of the Galois group of
the splitting field of $\Psi _{k}(x)$ over $%
\mathbb{Q}
$ and taking absolute values, we get 
\begin{equation*}
y^{m}=\left\vert \frac{(2\alpha _{i}-2)\left( \alpha _{i}-1\right) }{%
(k+1)\alpha _{i}^{2}-3k\alpha _{i}+k-1}\alpha _{i}^{n}\right\vert
\end{equation*}%
for some $i>1,$ where $\alpha =\alpha _{1},\alpha _{2},\ldots ,\alpha _{k}$
are the roots of $\Psi _{k}(x)$. As $|\alpha _{i}|<1$, using the inequality (%
\ref{3.4}), we can write the inequality%
\begin{eqnarray*}
y^{m} &=&\left\vert 2\alpha _{i}-2\right\vert \left\vert \frac{\left( \alpha
_{i}-1\right) }{(k+1)\alpha _{i}^{2}-3k\alpha _{i}+k-1}\right\vert
\left\vert \alpha _{i}\right\vert ^{n} \\
&<&4.
\end{eqnarray*}%
Hence, we have $y^{m}<4,$ which is impossible as $m,y\geq 2.$ Therefore $%
\Lambda _{1}\neq 0.$ Moreover, since $h(y)=\log y,$ $h(\gamma _{2})=\dfrac{%
\log \alpha }{k}<\dfrac{\log 3}{k}$ by (\ref{2.1}), we can take $%
A_{1}:=k\log y,~A_{2}:=\log 3.$ Now, let find the approximate value of $%
h((2\alpha -2)g_{k}(\alpha )).$ We know that $h(g_{k}(\alpha )<5\log k$ by
Lemma \ref{L7}. Besides, since the minimal polynomial of $\alpha -1$ over
the integers is 
\begin{equation*}
(x+1)^{k}-2(x+1)^{k-1}-(x+1)^{k-1}-\cdots -(x+1)-1,
\end{equation*}%
it can be seen that 
\begin{equation*}
h(\alpha -1)=\frac{1}{k}\left( \log 1+\dsum\limits_{i=1}^{k}\log \left( \max
\left\{ |\alpha _{i}-1|,1\right\} \right) \right) \leq \log 2.
\end{equation*}%
Consequently, we have 
\begin{eqnarray*}
h((2\alpha -2)g_{k}(\alpha )) &\leq &h(2)+h(\alpha -1)+h(g_{k}(\alpha )) \\
&<&\log 2+\log 2+5\log k<8\log k.
\end{eqnarray*}%
Hence, we can take $A_{3}:=8k\log k.$ Also, since $m\leq 1.73n,$ it follows
that $B:=1.73n.$ Thus, taking into account the inequality (\ref{3.3}) and
using Theorem \ref{T2}, we obtain{\small 
\begin{equation*}
\dfrac{1.82}{\alpha ^{n}}>\left\vert \Lambda _{1}\right\vert >\exp \left(
-C\cdot k^{2}(1+\log k)(1+\log \left( 1.73n\right) )\left( k\log y\right)
\left( \log 3\right) \left( 8k\log k\right) \right)
\end{equation*}%
}and so 
\begin{equation}
n\log \alpha -\log (1.82)<C\cdot k^{2}\cdot 3\log k\cdot 3\log n\cdot \left(
k\log y\right) \left( \log 3\right) \left( 8k\log k\right) ,  \label{12}
\end{equation}%
where $C=1.4\cdot 30^{6}\cdot 3^{4.5}$ and we have used the fact that $%
1+\log k<3\log k$ for $k\geq 2$ and $1+\log \left( 1.73n\right) <3\log n$
for $n\geq 3.$ From the inequality (\ref{12}), a quick computation with
Mathematica yields \textit{\ }%
\begin{equation*}
n<1.64\cdot 10^{13}\cdot k^{4}\cdot (\log k)^{2}\cdot \log y\cdot \log n.
\end{equation*}%
Thus, the proof is completed.
\end{proof}

\begin{theorem}
Let $2\leq y\leq 100$. Then all solutions $(n,m,y)$ of Diophantine equation 
\emph{(\ref{1.1})} are given by $(n,m,y)=(3,2,4),(3,4,2)$ with $k\geq 3.$
\end{theorem}

\proof%
Assume that Diophantine equation (\ref{1.1})\emph{\ }is satisfied for $2\leq
y\leq 100.$ If $1\leq n\leq k,$ then we have $Q_{n}^{(k)}=2F_{2n}=y^{m}$ by
Lemma \ref{L4} (c). From here, by Theorem 2 given in \cite{patel}, $%
F_{2n}=2^{m-1}(y/2)^{m}$ implies that $(n,m,y)=(3,2,4),(3,4,2).$ Now we
assume that $n\geq k+1.$ If $k=2,$ then $n\geq 3$ and we have $Q_{n}=y^{m}.$
By Theorem \ref{ty}, the equation $Q_{n}=y^{m}$ has no solutions in positive
integers $n\geq 3$ and $m\geq 2.$ Therefore, assume that $k\geq 3.$ In this
case $n\geq 4.$ Also, since $y\leq 100$, by (\ref{3.5}), we get \textit{\ }%
\begin{equation*}
n<7.56\cdot 10^{13}\cdot k^{4}\cdot (\log k)^{2}\cdot \log n.
\end{equation*}%
Rearranging the last inequality as 
\begin{equation*}
\frac{n}{\log n}<7.56\cdot 10^{13}\cdot k^{4}\cdot (\log k)^{2}
\end{equation*}%
and using the fact that 
\begin{equation*}
\text{if }A\geq 3\text{ and }\frac{n}{\log n}<A,\text{ then }n<2A\log A,
\end{equation*}%
we obtain 
\begin{eqnarray}
n &<&15.12\cdot 10^{13}\cdot k^{4}\cdot (\log k)^{2}\cdot \log \left(
7.56\cdot 10^{13}\cdot k^{4}\cdot (\log k)^{2}\right)  \label{3.6} \\
&<&15.12\cdot 10^{13}\cdot k^{4}\cdot (\log k)^{2}\cdot (32+4\log k+2\log
(\log k))  \notag \\
&<&15.12\cdot 10^{13}\cdot k^{4}\cdot (\log k)^{2}\cdot 34\log k  \notag \\
&<&5.141\cdot 10^{15}\cdot k^{4}\cdot (\log k)^{3},  \notag
\end{eqnarray}%
where we have used the fact that $32+4\log k+2\log (\log k)<34\log k$ for
all $k\geq 3$. Let $k\in \lbrack 3,510].$ Then, we get $n<2.831\cdot 10^{28}$
by (\ref{3.6}). Now, let us try to reduce the upper bound on $n$ applying
Lemma \ref{L2}. Let%
\begin{equation*}
z_{1}:=m\log y-n\log \alpha +\log \left( ((2\alpha -2)g_{k}(\alpha
))^{-1}\right)
\end{equation*}%
\bigskip and $x=e^{z_{1}}-1$. Then, from (\ref{3.3}), it is seen that 
\begin{equation*}
\left\vert x\right\vert =\left\vert e^{z_{1}}-1\right\vert <\frac{1.82}{%
\alpha ^{n}}<0.12
\end{equation*}%
for $n\geq 4.$ Choosing $a:=$ $0.12,$ we get the inequality%
\begin{equation*}
|z_{1}|=\left\vert \log (x+1)\right\vert <\frac{-\log (1-0.12)}{(0.12)}\cdot 
\frac{1.82}{\alpha ^{n}}<\frac{1.94}{\alpha ^{n}}
\end{equation*}%
by Lemma \ref{L3}. Thus, it follows that 
\begin{equation*}
0<\left\vert m\log y-n\log \alpha +\log \left( ((2\alpha -2)g_{k}(\alpha
))^{-1}\right) \right\vert <\frac{1.94}{\alpha ^{n}}.
\end{equation*}%
Dividing this inequality by $\log \alpha ,$ we get 
\begin{equation}
0<|m\gamma -n+\mu |<A\cdot B^{-w},  \label{3.7}
\end{equation}%
where 
\begin{equation*}
\gamma :=\dfrac{\log y}{\log \alpha },~\mu :=\dfrac{\log \left( ((2\alpha
-2)g_{k}(\alpha ))^{-1}\right) }{\log \alpha },~A:=1.94,~B:=\alpha \text{,
and }w:=n.
\end{equation*}%
It can be easily seen that $\dfrac{\log y}{\log \alpha }$ is irrational. If
it were not, then we could write $\dfrac{\log y}{\log \alpha }=\dfrac{b}{a}$
for some positive integers $a$ and $b$. This implies that $y^{a}=\alpha
^{b}. $ Conjugating this equality by some automorphisim belonging to the
Galois group of the splitting field of $\Psi _{k}(x)$ over $%
\mathbb{Q}
$ and taking absolute values, we get $y^{a}=|\alpha _{i}|^{b}$ for any $i>1.$
This is impossible since $|\alpha _{i}|<1$ and $y\geq 2.$ If we take%
\begin{equation*}
M:=4.9\cdot 10^{28},
\end{equation*}%
which is an upper bound on $m$ since $m\leq 1.73n<4.9\cdot 10^{28}$, we
found that $q_{81},$ the denominator of the $81$ th convergent of $\gamma $
exceeds $6M.$ Furthermore, a quick computation with Mathematica gives us
that the value%
\begin{equation*}
\dfrac{\log \left( Aq_{81}/\epsilon \right) }{\log B}
\end{equation*}%
is less than $144.6$ for all $k\in \left[ 3,510\right] $. So, if the
inequality (\ref{3.7}) has a solution, then \textit{\ } 
\begin{equation*}
n<\dfrac{\log \left( Aq_{81}/\epsilon \right) }{\log B}<144.6,
\end{equation*}%
which shows that $n\leq 144.$ In this case, $m\leq 249$ by (\ref{3.1}). A
quick computation with Mathematica gives us that the equation $%
Q_{n}^{(k)}=y^{m}$ with $2\leq y\leq 100$ has no solutions for $n\in \left[
4,144\right] ,m\in \lbrack 2,249]$ and $k\in \left[ 3,510\right] .$ Thus,
this completes the analysis in the case $k\in \left[ 3,510\right] $. From
now on, we can assume that $k>510.$ Then we can see from (\ref{3.6}) that
the inequality%
\begin{equation}
n<5.141\cdot 10^{15}\cdot k^{4}\cdot (\log k)^{3}<\varphi ^{k/2-2}<\varphi
^{k/2}  \label{x}
\end{equation}%
holds for $k>510.$

Now, let $\lambda >0$ be such that $\alpha +\lambda =\varphi ^{2}.$ By Lemma %
\ref{L1} (b), we obtain 
\begin{equation*}
\lambda =\varphi ^{2}-\alpha <\varphi ^{2}-\varphi ^{2}(1-\varphi
^{-k})=\varphi ^{-k+2},
\end{equation*}%
that is, 
\begin{equation}
\lambda <\dfrac{1}{\varphi ^{k-2}}.  \label{3.8}
\end{equation}%
Also $2\alpha -2=2\varphi ^{2}-2\lambda -2=2\varphi -2\lambda ,$ and so%
\begin{equation*}
2\varphi >2\alpha -2>2\varphi -\dfrac{2}{\varphi ^{k-2}}.
\end{equation*}%
Moreover, 
\begin{eqnarray*}
\alpha ^{n} &=&\left( \varphi ^{2}-\lambda \right) ^{n}=\varphi ^{2n}(1-%
\frac{\lambda }{\varphi ^{2}})^{n} \\
&=&\varphi ^{2n}e^{n\log (1-\frac{\lambda }{\varphi ^{2}})}\geq \varphi
^{2n}e^{-n\lambda }\geq \varphi ^{2n}(1-n\lambda ) \\
&>&\varphi ^{2n}\left( 1-\dfrac{n}{\varphi ^{k-2}}\right) ,
\end{eqnarray*}%
where we have used the facts that $\log (1-x)\geq -\varphi ^{2}x$ for $%
0<x<0.906,$ and $e^{-x}>1-x$ for all $x\in 
\mathbb{R}
\backslash \{0\}.$ Thus,%
\begin{equation*}
\alpha ^{n}>\varphi ^{2n}-\dfrac{n\varphi ^{2n}}{\varphi ^{k-2}}>\varphi
^{2n}-\dfrac{\varphi ^{2n}}{\varphi ^{k/2}}
\end{equation*}%
by (\ref{x}). In this case, 
\begin{eqnarray}
\left( 2\alpha -2\right) \alpha ^{n} &>&\left( 2\varphi -\dfrac{2}{\varphi
^{k-2}}\right) \left( \varphi ^{2n}-\dfrac{\varphi ^{2n}}{\varphi ^{k/2}}%
\right)  \label{3.21} \\
&=&2\varphi ^{2n+1}-2\varphi ^{2n}\left( \dfrac{\varphi }{\varphi ^{k/2}}+%
\dfrac{1}{\varphi ^{k-2}}-\dfrac{1}{\varphi ^{3k/2-2}}\right)  \notag \\
&>&2\varphi ^{2n+1}-\dfrac{4\varphi ^{2n}}{\varphi ^{k/2}},  \notag
\end{eqnarray}%
where we have used the fact that 
\begin{equation*}
\dfrac{\varphi }{\varphi ^{k/2}}+\dfrac{1}{\varphi ^{k-2}}-\dfrac{1}{\varphi
^{3k/2-2}}<\dfrac{2}{\varphi ^{k/2}}
\end{equation*}%
for $k>510.$ Since $\alpha <\varphi ^{2},$ it follows that%
\begin{equation*}
\left( 2\alpha -2\right) \alpha ^{n}<2\varphi \cdot \varphi ^{2n}<2\varphi
^{2n+1}+\dfrac{4\varphi ^{2n}}{\varphi ^{k/2}}
\end{equation*}%
and so we have 
\begin{equation}
\left\vert \left( 2\alpha -2\right) \alpha ^{n}-2\varphi ^{2n+1}\right\vert <%
\dfrac{4\varphi ^{2n}}{\varphi ^{k/2}}.  \label{3.9}
\end{equation}%
Let us consider 
\begin{equation*}
g_{k}(x)=\dfrac{x-1}{(k+1)x^{2}-3kx+k-1},
\end{equation*}%
defined in (\ref{0.5}) as a function of a real variable$.$ By the Mean-Value
Theorem, we can say that there exist some $\theta \in (\alpha ,\varphi ^{2})$
such that 
\begin{equation}
g_{k}^{^{\prime }}(\theta )=\dfrac{g_{k}(\alpha )-g_{k}(\varphi ^{2})}{%
\alpha -\varphi ^{2}}.  \label{3.10}
\end{equation}%
Calculating $g_{k}^{^{\prime }}(\theta )$ and using Lemma \ref{L1}, the
inequalities (\ref{0.6}) and (\ref{0.12}), we get 
\begin{eqnarray*}
\left\vert g_{k}^{^{\prime }}(\theta )\right\vert &=&\left\vert \frac{%
(k+1)\theta ^{2}-2k\theta -2\theta +2k+1}{\left( (k+1)\theta ^{2}-3k\theta
+k-1\right) ^{2}}\right\vert \\
&<&1+\frac{(k-2)\theta +k+2}{(k+1)\theta ^{2}-3k\theta +k-1} \\
&=&1+g_{k}(\theta )\frac{(k-2)\theta +k+2}{\theta -1} \\
&=&1+g_{k}(\alpha )\left( k-2+\frac{2k}{\theta -1}\right) \\
&<&1+\left( \frac{k-2}{2}+\frac{k}{\theta -1}\right) \\
&<&\frac{3k}{2}.
\end{eqnarray*}%
Hence, from (\ref{3.10}), it follows that%
\begin{equation*}
\left\vert g_{k}(\alpha )-g_{k}(\varphi ^{2})\right\vert =\left\vert \alpha
-\varphi ^{2}\right\vert \left\vert g_{k}^{^{\prime }}(\theta )\right\vert
=\lambda \left\vert g_{k}^{^{\prime }}(\theta )\right\vert ,
\end{equation*}%
which implies that 
\begin{equation}
\left\vert g_{k}(\alpha )-g_{k}(\varphi ^{2})\right\vert <\frac{3k/2}{%
\varphi ^{k-2}}<\frac{4k}{\varphi ^{k}}  \label{3.11}
\end{equation}%
by (\ref{3.8}).

Now let us record what we made.

\begin{lemma}
\label{L8}Let $k>510$ and let $\alpha $ be dominant root of $\Psi _{k}(x).$
Let consider $g_{k}(x)$ defined in \emph{(\ref{0.5})} as a function of a
real variable. Then 
\begin{equation*}
g_{k}(\alpha )=g_{k}(\varphi ^{2})+\eta ,
\end{equation*}%
where $\left\vert \eta \right\vert <\frac{4k}{\varphi ^{k}}.$
\end{lemma}

Taking into account the inequality (\ref{3.9}) and Lemma \ref{L8}, we can
write 
\begin{equation}
\left( 2\alpha -2\right) \alpha ^{n}=2\varphi ^{2n+1}+\delta \text{ and }%
g_{k}(\alpha )=g_{k}(\varphi ^{2})+\eta  \label{3.12}
\end{equation}%
such that 
\begin{equation}
\left\vert \delta \right\vert <\dfrac{4\varphi ^{2n}}{\varphi ^{k/2}}\text{
and }\left\vert \eta \right\vert <\frac{4k}{\varphi ^{k}}.  \label{3.13}
\end{equation}%
Thus, since $g_{k}(\varphi ^{2})=\dfrac{1}{\varphi +2}$ by Lemma \ref{L1}
(c), it is seen that 
\begin{equation}
\left( 2\alpha -2\right) g_{k}(\alpha )\alpha ^{n}=\frac{2\varphi ^{2n+1}}{%
\varphi +2}+\frac{\delta }{\varphi +2}+2\varphi ^{2n+1}\eta +\eta \delta .
\label{3.14}
\end{equation}%
So, using (\ref{3.2}), (\ref{3.13}), and (\ref{3.14}), we obtain{\small 
\begin{eqnarray}
\left\vert y^{m}-\frac{2\varphi ^{2n+1}}{\varphi +2}\right\vert
&=&\left\vert \left( y^{m}-\left( 2\alpha -2\right) g_{k}(\alpha )\alpha
^{n}\right) +\frac{\delta }{\varphi +2}+2\varphi ^{2n+1}\eta +\eta \delta
\right\vert  \label{3.15} \\
&\leq &\left\vert y^{m}-\left( 2\alpha -2\right) g_{k}(\alpha )\alpha
^{n}\right\vert +\frac{\left\vert \delta \right\vert }{\varphi +2}+2\varphi
^{2n+1}\left\vert \eta \right\vert +\left\vert \eta \right\vert \left\vert
\delta \right\vert  \notag \\
&<&2+\dfrac{4\varphi ^{2n}}{\varphi ^{k/2}\left( \varphi +2\right) }+\frac{%
8k\varphi ^{2n+1}}{\varphi ^{k}}+\dfrac{16k\varphi ^{2n}}{\varphi ^{3k/2}}. 
\notag
\end{eqnarray}%
}Dividing both sides of the above inequality by $\dfrac{2\varphi ^{2n+1}}{%
\varphi +2},$ we get{\small 
\begin{eqnarray}
\left\vert y^{m}\varphi ^{-(2n+1)}\left( \frac{\varphi +2}{2}\right)
-1\right\vert &<&\frac{\varphi +2}{\varphi ^{2n+1}}+\dfrac{2/\varphi }{%
\varphi ^{k/2}}+\frac{4k\left( \varphi +2\right) }{\varphi ^{k}}+\dfrac{%
8k\left( \varphi +2\right) }{\varphi \cdot \varphi ^{3k/2}}  \label{3.16} \\
&<&\dfrac{0.001}{\varphi ^{k/2}}+\dfrac{1.24}{\varphi ^{k/2}}+\dfrac{0.005}{%
\varphi ^{k/2}}+\dfrac{0.005}{\varphi ^{k/2}}=\dfrac{1.251}{\varphi ^{k/2}},
\notag
\end{eqnarray}%
} where we have used the facts that 
\begin{equation*}
\frac{\varphi +2}{\varphi ^{2n+1}}<\dfrac{0.001}{\varphi ^{k/2}},~\frac{%
4k\left( \varphi +2\right) }{\varphi ^{k}}<\dfrac{0.005}{\varphi ^{k/2}}
\end{equation*}%
and%
\begin{equation*}
\dfrac{\left( 8k/\varphi \right) \left( \varphi +2\right) }{\varphi ^{3k/2}}<%
\dfrac{0.005}{\varphi ^{k/2}}
\end{equation*}%
for $k>510$. In order to use Theorem \ref{T2}, we take 
\begin{equation*}
\left( \gamma _{1},b_{1}\right) :=\left( y,m\right) ,~\left( \gamma
_{2},b_{2}\right) :=\left( \varphi ,-2n\right) ,~\left( \gamma
_{3},b_{3}\right) :=\left( \dfrac{\varphi +2}{2\varphi },1\right) .
\end{equation*}%
The number field containing $\gamma _{1},\gamma _{2}$ and $\gamma _{3}$ are $%
\mathbb{K}=\mathbb{Q}(\sqrt{5}),$ which has degree $D=2.$ We show that the
number 
\begin{equation*}
\Lambda _{1}:=y^{m}\varphi ^{-(2n+1)}\left( \frac{\varphi +2}{2}\right) -1
\end{equation*}%
is nonzero. Contrast to this, assume that $\Lambda _{1}=0.$ Then $%
y^{m}\left( \frac{\varphi +2}{2}\right) =\varphi ^{2n+1}$ and conjugating
this relation in $\mathbb{Q}(\sqrt{5}),$ we get $y^{m}\left( \frac{\beta +2}{%
2}\right) =\beta ^{2n+1},$ where $\beta =\frac{1-\sqrt{5}}{2}=\overline{%
\varphi }.$ From this, it is seen that 
\begin{equation*}
\dfrac{\varphi ^{2n+1}}{\varphi +2}=\dfrac{\beta ^{2n+1}}{\beta +2}<0,
\end{equation*}%
which is impossible since $\varphi >0.$ Therefore $\Lambda _{1}\neq 0.$
Moreover, since 
\begin{equation*}
h(\gamma _{1})=h(y)=\log y,h(\gamma _{2})=h(\varphi )\leq \frac{\log \varphi 
}{2}
\end{equation*}%
and 
\begin{equation*}
h(\gamma _{3})=h\left( \dfrac{\varphi +2}{2\varphi }\right) =h\left( \frac{%
\sqrt{5}}{2}\right) \leq \frac{\log 5}{2}
\end{equation*}%
by (\ref{2.3}), we can take $A_{1}:=2\log y,~A_{2}:=\log \varphi $ and $%
A_{3}:=\log 5.$ Also, since $m<1.73n,$ we can take $B:=2n.$ Thus, taking
into account the inequality (\ref{3.16}) and using Theorem \ref{T2}, we
obtain%
\begin{equation*}
\left( 1.251\right) \cdot \varphi ^{-k/2}>\left\vert \Lambda _{1}\right\vert
>\exp \left( C\cdot (1+\log 2n)\left( 2\log y\right) \left( \log \varphi
\right) \left( \log 5\right) \right) ,
\end{equation*}%
where $C=-1.4\cdot 30^{6}\cdot 3^{4.5}\cdot 2^{2}\cdot (1+\log 2)$. This
implies that 
\begin{equation*}
\frac{k}{2}\log \varphi -\log (1.251)<6.92\cdot 10^{12}\cdot (1+\log 2n)
\end{equation*}%
or 
\begin{equation}
k<3.31\cdot 10^{13}\cdot \log n,  \label{3.17}
\end{equation}%
where we have used the fact that $(1+\log 2n)<(2.3)\log n$ for $n\geq 4.$ On
the other hand, from (\ref{3.6}), we get 
\begin{eqnarray*}
\log n &<&\log \left( 5.141\cdot 10^{15}\cdot k^{4}\cdot (\log k)^{3}\right)
\\
&<&36.2+4\log k+3\log (\log k) \\
&<&38\log k
\end{eqnarray*}%
for $k\geq 3.$ So, from (\ref{3.17}), we obtain 
\begin{equation*}
k<3.31\cdot 10^{13}\cdot 38\log k,
\end{equation*}%
which implies that 
\begin{equation}
k<4.84\cdot 10^{16}.  \label{3.18}
\end{equation}%
To reduce this bound on $k$, we use Lemma \ref{L2}. Substituting this bound
of $k$ into (\ref{3.6}), we get $n<1.6\cdot 10^{87},$ which shows that $%
m<2.77\cdot 10^{87}.$

Now, let 
\begin{equation*}
z_{2}:=m\log y-(2n+1)\log \varphi +\log \left( \frac{\varphi +2}{2}\right)
\end{equation*}%
and $x:=1-e^{z_{2}}$. Then 
\begin{equation*}
|x|=\left \vert 1-e^{z_{2}}\right \vert <\dfrac{1.251}{\varphi ^{k/2}}
\end{equation*}%
by (\ref{3.16}) and so $|x|<0.1$ for $k>510$. Choosing $a:=$ $0.1,$ we get
the inequality

\begin{equation*}
|z_{2}|=\left \vert \log (x+1)\right \vert <\frac{\log (100/90)}{0.1}\cdot 
\dfrac{1.251}{\varphi ^{k/2}}<\frac{1.32}{\varphi ^{k/2}}
\end{equation*}%
by Lemma \ref{L3}. That is, 
\begin{equation*}
0<\left \vert m\log y-(2n+1)\log \varphi +\log \left( \frac{\varphi +2}{2}%
\right) \right \vert <\frac{1.32}{\varphi ^{k/2}}.
\end{equation*}%
Dividing both sides of the above inequality by $\log \varphi ,$ we obtain 
\begin{equation}
0<|m\gamma -(2n+1)+\mu |<A\cdot B^{-w},  \label{3.20}
\end{equation}%
where 
\begin{equation*}
\gamma :=\dfrac{\log y}{\log \varphi }\notin 
\mathbb{Q}
,~\mu :=\dfrac{\log \left( \frac{\varphi +2}{2}\right) }{\log \varphi }%
,~A:=2.75,~B:=\varphi \text{, and }w:=k/2.
\end{equation*}%
If we take $M:=2.77\cdot 10^{87}$, which is an upper bound on $m$, we found
that $q_{207},$ the denominator of the $207$ th convergent of $\gamma $
exceeds $6M.$ Furthermore, a quick computation with Mathematica gives us
that the value 
\begin{equation*}
\dfrac{\log \left( Aq_{207}/\epsilon \right) }{\log B}
\end{equation*}%
is less than $585.91.$ So, if the inequality (\ref{3.20}) has a solution,
then \textit{\ } 
\begin{equation*}
\frac{k}{2}<\dfrac{\log \left( Aq_{207}/\epsilon \right) }{\log B}\leq
585.91,
\end{equation*}%
which implies that $k\leq 1171.$ Hence, from (\ref{3.6}), we get $%
n<3.41\cdot 10^{30},$ which shows that $m<5.9\cdot 10^{30}.$ If we apply the
inequality (\ref{3.20}) to Lemma \ref{L2} again with $M:=5.9\cdot 10^{30},$
we found that $q_{69},$ the denominator of the $69$ th convergent of $\gamma 
$ exceeds $6M.$ After doing this, then a quick computation with Mathematica
show that in case the inequality (\ref{3.20}) has a solution, we get $k<505.$
This contradicts the fact that $k>510.$ This completes the proof.%
\endproof%

\end{document}